\documentclass[10pt]{amsart}
\usepackage{amssymb,amsbsy,amsmath,amsfonts,times, amscd}
\usepackage{latexsym,euscript,exscale}
\usepackage{graphicx,color} \usepackage{amsthm}
\usepackage{fancyhdr} \usepackage{url}
\usepackage{epigraph}\usepackage{lmodern}
\usepackage{mathrsfs}\usepackage{cancel}
\usepackage[toc,page]{appendix}\usepackage[T1]{fontenc}
\usepackage{mathtools}\usepackage[all,cmtip]{xy}
\usepackage{enumerate}\usepackage{setspace}
\usepackage{helvet}
\usepackage{float}

\numberwithin{equation}{section}
\newtheorem{theorem}{Theorem}
\newtheorem{thm}[theorem]{Theorem}

\newtheorem{cor}[theorem]{Corollary}
\newtheorem{lemma}[theorem]{Lemma}
\newtheorem{prop}[theorem]{Proposition}
\theoremstyle{definition}
\newtheorem{defn}[theorem]{Definition}
\newtheorem{problem}[theorem]{Problem}
\theoremstyle{remark}

\newcommand{\ds}{\displaystyle}
\newcommand{\eps}{\varepsilon}

\newcommand{\cB}{\mathcal{B}}
\newcommand{\cF}{\mathcal{F}}

\newcommand{\cA}{\mathcal{A}}

\newcommand{\cP}{\mathcal{P}}

\newcommand{\cE}{\mathcal{E}}

\newcommand{\SB}{\text{SB}}

\newcommand{\Q}{\mathbb{Q}}\newcommand{\emb}{\hookrightarrow}
\newcommand{\N}{\mathbb{N}}
\newcommand{\R}{\mathbb{R}}

\newcommand{\NN}{{\mathbb{N}^{<\mathbb{N}}}}
\newcommand{\WF}{\text{WF}}
\newcommand{\IF}{\text{IF}}
\newcommand{\Tr}{\text{Tr}}

\begin{document}

\title[Complexity of some classes of Banach spaces.]{On the  complexity of some inevitable classes of separable Banach spaces.}
\author{ B. M. Braga.}
\address{Department of Mathematics, Statistics, and Computer Science (M/C 249)\\
University of Illinois at Chicago\\
851 S. Morgan St.\\
Chicago, IL 60607-7045\\
USA}\email{demendoncabraga@gmail.com}
\date{}
\date{}

\begin{abstract}
In this paper, we study the descriptive complexity of some inevitable classes of Banach spaces. Precisely, as shown in \cite{Go}, every Banach space either contains a hereditarily indecomposable subspace or an unconditional basis, and, as shown in \cite{FR}, every Banach space either contains a minimal subspace or a continuously tight subspace. In these notes, we study the complexity of those inevitable classes as well as the complexity of containing a subspace in any of those classes.
\end{abstract}
\maketitle

\section{Introduction.}\label{intro}

Let $\SB=\{X\subset C(\Delta)\mid X\text{ is linear and closed}\}$ be our coding for the separable Banach spaces, and endow $\SB$ with the Effros-Borel structure (for definitions, see Section \ref{definition}). Given a separable Banach space $X\in\SB$, it is natural to ask about  the complexity of the isomorphism class $\langle X\rangle=\{Y\in\SB\mid Y\cong X\}$.  Is $\langle X\rangle$ Borel, analytic, coanalytic? This question has only been solved for some very specific Banach spaces. For example,  S. Kwapien showed (see \cite{Kw}, Proposition $3.1$) that a Banach space $X$ is isomorphic to a Hilbert space if, and only if, $X$ has both type and cotype equal $2$. As $\ell_2$ is the only separable Hilbert space (up to isomorphism), this gives us a characterization of separable Banach spaces which are  isomorphic to $\ell_2$ in terms of type and cotype. Using this characterization,  it is not hard to show that $\langle \ell_2\rangle$ is Borel (see \cite{B}, page $130$). In fact, $\langle \ell_2\rangle$ is the only known example of a Borel isomorphism class. We recall the following problem (see \cite{B}, Problem 2.9).

\begin{problem}
Let $X\in\SB$ be a separable Banach space whose isomorphism class $\langle X\rangle$ is Borel. Is $X$ isomorphic to $\ell_2$? 
\end{problem}

For some classical spaces, such as $L_p[0,1]$, with $p\neq 2$, Pelczynski's universal space $U$, and $C(\Delta)$, it had been shown that their isomorphism classes are analytic non Borel (see \cite{B}, page $130$ and Theorem $2.3$, and \cite{K}, Theorem $33.24$, respectively). 

The problem of classifying Banach spaces up to isomorphism is extremely complex. Indeed, considering the isomorphism relation $\cong$ as a subset of $\SB\times\SB$, we have that $\cong$ is a complete analytic equivalence relation (see \cite{FLR}, Theorem $31$). In a more precise sense, separable Banach spaces up to isomorphism may serve as a complete invariant for any other reasonable class of mathematical objects. Therefore, it makes sense to study easier problems such as classifying Banach spaces up to its subspaces. We have then another natural problem: given a Banach space $X\in\SB$, what can we say about the  complexity of $\text{C}_X=\{Y\in\SB\mid X\emb Y\}$? This question had been solved by B. Bossard (see \cite{B}, Corollary $3.3$).

\begin{thm}\textbf{(Bossard)}\label{nonborel}
Let $X\in\SB$. If $X$ is finite dimensional, then $\text{C}_X$ is Borel. If $X$ is infinite dimensional, then $\text{C}_X$ is complete analytic. 
\end{thm}

In \cite{Go2}, W. T. Gowers had started with a new classification theory for Banach spaces. Indeed, after Gowers and Maurey had solved the unconditional basis problem, and proved the existence of hereditarily indecomposable Banach spaces (see \cite{GM}), Gowers proved that every Banach space either contains a subspace with an unconditional basis, or a hereditarily indecomposable subspace (see \cite{Go}). Later, in \cite{Go2}, Gowers used Ramsey methods in order to refine his dichotomy. Many other dichotomies had been proved by V. Ferenczi and C. Rosendal in \cite{FR}, for example, their main dichotomy says that   every Banach space either contains a minimal subspace or a continuously tight subspace (see \cite{FR}, Theorem 3.13).

In these notes, we study the descriptive complexity of the inevitable classes above as well as the complexity of containing a subspace in any of those inevitable classes. The table below summarizes the lower and upper bounds known for the complexity of those classes. In the table below, an asterisk before the complexity indicates that this bound was already known before this paper. All the other bounds are either computed in this paper or are given by 
trivial computations. 

\bigskip
\begin{center}
\begin{tabular}{ lll }
\hline

Classes of Banach spaces& Lower bound&Upper bound\\
\hline
Minimal spaces&&$\ \ \ \ \  \Delta^1_2$ \\
Containing a minimal subspace\ \ \ \ \ &$\ \ \ \ \ \Sigma^1_1$-hard  &$(*)\  \Sigma^1_2$   \\
Continuously tight spaces&$\ \ \ \ \  \Pi^1_1$-hard& $\ \ \ \ \ \Sigma^1_2$\\
Containing a continuously tight subspace&$\ \ \ \ \ \Sigma^1_1$-hard&$\ \ \ \ \  \Sigma^1_2$ \\
HI spaces&$\ \ \ \ \ \Pi^1_1$-hard \ \ \ &$\ \ \ \ \  \Pi^1_1$\\
Containing a HI subspace&$\ \ \ \ \ \Sigma^1_1$-hard&$\ \ \ \ \ \Sigma^1_2$\\
Spaces with unconditional basis & &$\ \ \ \ \     \Sigma^1_1$\\
Containing unconditional basis&$\ \ \ \ \  \Sigma^1_1$-hard&$\ \ \ \ \  \Sigma^1_1$\\
\hline
\end{tabular}
\end{center}
\bigskip

In particular, we completely compute the descriptive complexity of the class of hereditarily indecomposable spaces, and the class of spaces containing an unconditional basis, by showing those classes are complete coanalytic and complete analytic, respectively. 

The properties of being a continuously tight space and of having an unconditional basis are more properties about basis then about the spaces themselves. Therefore, it is more natural to ask  what is the complexity of the set of continuously tight basis, and the complexity of the set of unconditional basis. We have the following.

\bigskip
\begin{center}
\begin{tabular}{ lll }
\hline

Classes of basis& Lower bound&Upper bound\\
\hline
Continuously tight basis& $\ \ \ \ \ \Pi^1_1$-hard\ \ \ \ &$\ \ \ \ \  \Sigma^1_2$ \\
Unconditional basis\ \ \ \ \ \ & &\ \ \ \ \  Borel\\
\hline
\end{tabular}
\end{center}
\bigskip

This paper is organized in the following way. In Section \ref{definition}, we give all the basic background on both descriptive set theory and the theory of Banach spaces necessary for these notes. In Section \ref{sectionlema}, we prove a general lemma for Banach spaces with an unconditional basis $(x_{s})_{s\in\theta}$ indexed by a tree $\theta\in\WF$. In Section \ref{sectionlp}, we work with $\ell_p$-Baire sums of basic sequences as a tool to compute lower bounds for the complexity of classes of Banach spaces.  Then, Theorem \ref{lema} gives a sufficient condition for the set of separable Banach spaces containing a subspace in a given class to be $\Sigma^1_1$-hard (in particular, non Borel). 

Section \ref{newnew} is dedicated to the computation of the complexities in the tables above. We apply  Theorem \ref{lema} to the class of spaces containing a  hereditarily indecomposable subspace (Subsection \ref{masha}, Corollary \ref{hhii}), and to the class of spaces containing a continuously tight subspace (Subsection \ref{sectiontight}, Corollary \ref{cu}), obtaining a lower bound for the complexity of those classes.

In order to obtain the results above, we rely on the method of $\ell_p$-Baire sums of separable Banach spaces (see Lemma \ref{arroto}). However,  this method does not allow us to obtain any information about the complexity of the class of spaces (i) containing an unconditional basis, (ii) containing a minimal subspace, or (iii) containing a continuously tight subspace. Hence, in Subsection \ref{sectionminimal},  we define a parameterized version of Tsirelson space in order to show that the class of spaces containing a minimal  subspace is $\Sigma^1_1$-hard. Such parametrization will show that it is impossible to Borel separate the set of separable Banach spaces containing $c_0$ from the set of Tsirelson-saturated Banach spaces (see Theorem \ref{lemaminimal}). In Subsection \ref{sectiontightmesmo}, we use the methods of Subsection \ref{sectionminimal} in order to show that the class of continuously tight spaces is $\Pi^1_1$-hard (Theorem \ref{tighttighthu}).

At last, in Subsection \ref{hiGM}, we follow Argyros presentation  of how to construct HI extensions of ground norms (\cite{AT}, Chapters II and III) in order to  show that the set of hereditarily indecomposable separable Banach spaces is  complete coanalytic (Theorem \ref{hhhhiiii}). We also obtain that the set of hereditary indecomposable Banach spaces cannot be Borel separated from the set of Banach spaces containing $\ell_1$.  By the same methods, in Subsection \ref{ubub}, we show  that the class of Banach spaces containing an unconditional basis is complete analytic (Theorem \ref{ubca}).

\begin{problem}
For all the classes in the tables above, we have only completely computed the descriptive complexity of the set of HI spaces, and the set of spaces containing an unconditional basis. What are the exact complexities of the remaining classes?
\end{problem}

At last, we would like to make a small remark. In \cite{FR}, the authors actually present Theorem 1.1 as their main dichotomy, which says that every Banach space contains either a minimal subspace or a tight subspace. This dichotomy does not explicitly say anything about continuously tight spaces. However,  their proof for Theorem 1.1 shows that any Banach space $X$ must contain either a minimal Banach subspace or a continuously tight Banach subspace (see \cite{FR}, Theorem 3.13).

 We have two reasons for choosing to deal with continuously tight spaces instead of tight spaces, (i) we easily get better estimates for the upper bound of continuously tight spaces and spaces containing a continuously tight subspace, (ii) Rosendal had shown (see \cite{R}, Appendix) that, for minimal Banach spaces not containing $c_0$, there exists a notion of being continuously minimal, and this notion coincides with being minimal (see Subsection \ref{outrodia}).

\section{Background.}\label{definition}

A separable metric space is said to be a \emph{Polish space} if there exists an equivalent metric in which it is complete. A continuous image of a Polish space into another Polish 
space is called an \emph{analytic} set and a set whose complement is analytic is called \emph{coanalytic}. A measure space $(X,\mathcal{A})$, where $X$ is a set and $\mathcal{A}$ is a 
$\sigma$-algebra of subsets of $X$, is called a \emph{standard Borel space} if there exists a Polish topology on this set whose Borel $\sigma$-algebra coincides with $\mathcal{A}$. We 
define Borel, analytic and coanalytic sets in standard Borel spaces by saying that these are the sets that, by considering a compatible Polish topology, are Borel, analytic, and 
coanalytic, respectively. Observe that this is well defined, i.e., this definition does not depend on the Polish topology itself but on its Borel structure. A function between two 
standard Borel spaces is called \emph{Borel measurable} if the inverse image of each Borel subset of its codomain is Borel in its domain. We usually refer to Borel measurable functions 
just as Borel functions. Notice that, if you consider a Borel function between two standard Borel spaces, its inverse image of analytic sets (resp. coanalytic) is analytic (resp. 
coanalytic) (see \cite{S}, Proposition $1.3$, page $50$). 

Given a Polish space $X$, the set of analytic (resp. coanalytic) subsets of $X$ is denoted by $\Sigma^1_1(X)$ (resp. $\Pi^1_1(X)$). We usually omit $X$, and simply write $\Sigma^1_1$, or $\Pi^1_1$. Hence, the terminology $\Sigma^1_1$-set (resp. $\Pi^1_1$-set) is used to refer to analytic sets (resp. coanalytic sets). We define the \emph{projective hierarchy} by induction of $n$. We say a subset of a standard Borel space  is  $\Pi^1_n$ if it is the complement of a $\Sigma^1_n$-set, and we say that a subset is $\Sigma^1_{n+1}$ if it is a Borel image of a $\Pi^1_n$-set (see \cite{D}, Chapter 1). For each $n\in\N$, we denote by $\Delta^1_n$ the subsets of a standard Borel space which are both $\Sigma^1_n$ and $\Pi^1_n$, i.e., $\Delta^1_n=\Sigma^1_n\cap \Pi^1_n$.

Let $X$ be a standard Borel space. An analytic (resp. coanalytic) subset $A\subset X$ is said to be \emph{complete analytic} (resp. \emph{complete coanalytic}) if for every standard Borel space $Y$ and every analytic subset $B\subset Y$  (resp. coanalytic), there exists a Borel function $f:Y\to X$ such that $f^{-1}(A)=B$. This function is called a \emph{Borel reduction} from $B$ to $A$, and $B$ is said to be \emph{Borel reducible} to $A$. 

Let $X$ be a standard Borel space. We call a subset $A\subset X$ \emph{$\Sigma^1_n$-hard} (resp. \emph{$\Pi^1_n$-hard}) if for all standard Borel space $Y$ and all $\Sigma^1_n$-set $B\subset Y$  (resp. $\Pi^1_n$-set) there exists a Borel reduction from $B$ to $A$. Therefore, saying that a set $A\subset X$ is $\Sigma^1_n$-hard (resp. $\Pi^1_n$-hard)  means that $A$ is at least as complex as  any $\Sigma^1_n$-set (resp. $\Pi^1_n$-set) in the projective hierarchy. With this terminology we have that $A\subset X$ is  $\Sigma^1_n$-complete (resp. $\Pi^1_n$-complete) if, and only if, $A$ is  $\Sigma^1_n$-hard (resp. $\Pi^1_n$-hard) and $\Sigma^1_n$ (resp. $\Pi^1_n$). 

As there exist analytic non Borel (resp. coanalytic non Borel) sets we have that  $\Sigma^1_1$-hard (resp. $\Pi^1_1$-hard) sets are non Borel. Also, if $X$ is a standard Borel space, $A\subset X$, and there exists a Borel reduction from a $\Sigma^1_1$-hard (resp. $\Pi^1_1$-hard) subset $B$ of a standard Borel space $Y$ to $A$, then $A$ is  $\Sigma^1_1$-hard (resp. $\Pi^1_1$-hard). If $A$ is also analytic (resp. coanalytic), then $A$ is $\Sigma^1_1$-complete (resp. $\Pi^1_1$-complete). 

Consider a Polish space $X$ and let $\cF(X)$ be the set of all its non empty closed sets. We endow $\cF(X)$  with the \emph{Effros-Borel structure}, i.e., the $\sigma$-algebra 
generated by

$$\{F\subset X \mid F\cap U\neq \emptyset\},$$\\
\noindent where $U$ varies among the open sets of $X$. It can be shown that $\cF(X)$, endowed with the Effros-Borel structure, is a standard Borel space (see \cite{K}, Theorem $12.6$).  The following well known lemma (see 
\cite{K}, Theorem $12.13$) will be crucial in some of our proofs.

\begin{lemma}\textbf{(Kuratowski and Ryll-Nardzewski selection principle)}\label{lll}
Let $X$ be a Polish space. There exists a sequence of Borel functions $(S_n)_{n\in\N}:\cF(X)\to X$ such that $\{S_n(F)\}_{n\in\N}$ is dense in $F$, for all closed $F\subset X$.
\end{lemma}

In these notes we will only work with separable Banach spaces. We denote the closed unit ball of a Banach space $X$ by $B_X$. It is well 
known that every separable Banach space is isometrically isomorphic to a closed linear subspace of $C(\Delta)$ (see \cite{K}, page $79$), where $\Delta$ denotes the Cantor set. Therefore, 
$C(\Delta)$ is called \emph{universal} for the class of separable Banach spaces and we can code the class of separable Banach spaces, denoting it by $\text{SB}$, by 
$\text{SB}=\{X\subset C(\Delta)\mid  X \text{ is a closed }\allowbreak\text{linear subspace of }\allowbreak C(\Delta)\}$. As $C(\Delta)$ is clearly a Polish space we can endow 
$\cF(C(\Delta))$ with the Effros-Borel structure. It can be shown that $\SB$ is a Borel set in $\cF(C(\Delta))$ and hence it is also a standard Borel space (see \cite{D}, Theorem $2.2$). It now makes
sense to wonder if specific classes of separable Banach spaces are Borel or not (in our coding $\SB$). 

Throughout these notes we will denote by $\{S_n\}_{n\in\N}$ the sequence of Borel functions $S_n:\SB\to C(\Delta)$ given by \text{Lemma \ref{lll}}, with $X=C(\Delta)$ (more precisely, the restriction of those functions to $\SB$).

Consider the standard Borel space $C(\Delta)^\N$, and let 

$$\cB=\{(x_j)_{j\in\N}\in C(\Delta)^\N\mid (x_j)_{j\in\N}\text{ is a basic sequence}\}.$$\\
\noindent It is easy to see that $\cB$ is a Borel set. Therefore, we have that $\cB$ is a standard Borel space, and we can code all the basic sequences as elements of $\cB$. We can now wonder about the descriptive complexity of some specific classes of basis.

Let $(e_n)_{n\in\N}$ be a basic sequence in a Banach space $X\in\SB$. By a standard Skolem hull construction, there is a countable subfield $\mathbb{F}$ of $\R$ containing the rationals such that for any finite linear combination

$$\lambda_0 e_0+...+\lambda_n e_n,$$\\
\noindent with $n\in\N$, and $\lambda_1,...,\lambda_n\in\mathbb{F}$, we have $\|\lambda_0 e_0+...+\lambda_n e_n\|\in\mathbb{F}$. By working with $\mathbb{F}$ instead of $\Q$, we guarantee that any $\mathbb{F}$-linear combination of $(e_n)_{n\in\N}$ can be normalized and remain a $\mathbb{F}$-linear combination of $(e_n)_{n\in\N}$. Clearly, the $\mathbb{F}$-span of $(e_n)_{n\in\N}$ is dense in $\overline{\text{span}}\{e_n\}$.

Let $D$ be the set of normalized blocks of $(e_n)_{n\in\N}$, i.e., the set of all $\lambda_0 e_0+...+\lambda_n e_n,$ with $n\in\N$, and $\lambda_1,...,\lambda_n\in\mathbb{F}$. Clearly, $D$ is countable. A \emph{block sequence} of $(e_n)_{n\in\N}$ is a sequence $(y_n)_{n\in\N}$ such that, (i) there exist increasing sequences of natural numbers $(p_n)_{n\in\N}$ and $(q_n)_{n\in\N}$ such that $p_i\leq q_i<p_{i+1}$, for all $i\in\N$, and (ii) for all $n\in\N$

$$y_n=\sum_{j=p_n}^{q_n}a_je_j,$$\\
\noindent for some sequence $(a_n)_{n\in\N}$ of real numbers. We define a  \emph{finite block sequence} analogously. Denote by $bb(e_n)$ the set of normalized block sequences with the coefficients $(a_n)_{n\in\N}$ in $\mathbb{F}$. Endowing $D$ with the discrete topology, we can see $bb(e_n)$ as a closed subset of $D^\N$, so $bb(e_n)$ is a standard Borel space. Denote by $fbb(e_n)$ the set of normalized finite block sequences with the coefficients $(a_n)_{n\in\N}$ in $\mathbb{F}$. Letting $D^{<\N}=\cup_nD^n$,  we can see $fbb(e_n)$ as a closed subset of $D^{<\N}$. Hence, $fbb(e_n)$ is a standard Borel space.

Let $X$ and $Y$ be Banach spaces. We write  $X\hookrightarrow Y$ if $X$ can be linearly embedded into $Y$. If $K>0$, we write $X\hookrightarrow_K Y$ if $X$ can be $K$-embedded in $Y$, i.e., if there exists an embedding $T:X\to Y$ such that $\|T\|\|T^{-1}\|\leq K$. If $(x_n)_{n\in\N}$ and $(y_n)_{n\in\N}$ are two sequences in Banach spaces, we write $(x_n)_{n\in\N}\approx (y_n)_{n\in\N}$ if $(x_n)_{n\in\N}$ is equivalent to $(y_n)_{n\in\N}$, i.e., if the map $x_n\mapsto y_n$ induces a linear isomorphism between $\overline{\text{span}}\{x_n\}$ and $\overline{\text{span}}\{y_n\}$. Similarly, if $K>0$, we write $(x_n)_{n\in\N}\approx_K (y_n)_{n\in\N}$ if the induced isomorphism if a $K$-isomorphism.

Denote by $\NN$ the set of all finite tuples of natural numbers plus the empty set. Given $s=(s_0,...,s_{n-1}),\allowbreak t=(t_0,...,t_{m-1})\in\NN$ we say that the length of $s$ is 
$|s|=n$, $s_{|i}=(s_0,...,s_{i-1})$, for all $i\in\{1,...,n\}$, $s_0=\{\emptyset\}$, $s\preceq t$ \emph{iff} $n\leqslant m$ and $s_i=t_i$, for all $i\in\{0,...,n-1\}$, i.e., if $t$ is an 
extension of $s$. We define $s<t$ analogously. Define the concatenation of $s$ and $t$ as $s^\smallfrown t=(s_0,...,s_{n-1},t_0,...,t_{m-1})$.  A subset $T$ of $\NN$ is called a 
\emph{tree} if $t\in T$ implies $t_{|i}\in T$, for all $i\in\{0,...,|t|\}$. We denote the set of trees on $\N$ by $\text{Tr}$. A subset $I$ of a tree $T$ is called a segment if $I$ is 
completely ordered and if $s,t\in I$ with $s\preceq t$, then $l\in I$, for all $l\in T$ such that $s\preceq l\preceq t$. Two segments $I_1,\ I_2$ are called completely incomparable if neither 
$s\preceq t$ nor $t\preceq s$ hold if $s\in I_1$ and $t\in I_2$.

As $\NN$ is countable, $2^\NN$ (the power set of $\NN$) is Polish with its standard product topology. If we think about $\text{Tr}$ as a subset of $2^\NN$,  it is easy to see that $\text{Tr}$ is a $G_\delta$ set in $2^\NN$. Thus, it is also 
 Borel in $2^\NN$. As $\text{Tr}$ is Borel in the Polish space $2^{\NN}$, we have that $\text{Tr}$ is a standard Borel space. A $\beta\in\N^\N$ is called a \emph{branch} of 
 a tree $T$ if $\beta_{|i}\in T$, for all $i\in\N$, where $\beta_{|i}$ is defined analogously as above. We call a tree $T$ \emph{well-founded} if $T$ has no branches and 
 \emph{ill-founded} otherwise, we denote the set of well-founded and ill-founded trees by $\WF$ and $\IF$, respectively. It is well known that $\WF$ is a complete coanalytic set of 
 $\text{Tr}$, hence $\IF$ is complete analytic (see \cite{K}, Theorem $27.1$).

\section{A Lemma.}\label{sectionlema}

In this section, we prove a basic lemma that will be essential in many of the main results of this paper. Fix a compatible enumeration of $\NN$, i.e., a sequence $(s_n)_{n\in\N}$ in $\NN$ such that $s_n\preceq s_m$ implies $n\leqslant m$ and for all $s\in\NN$ there exists $n\in\N$ such that $s_n=s$. With this enumeration in mind, if $\theta\in\Tr$, we say that a sequence $(x_s)_{s\in\theta}$ is a basis for a given Banach space $X$, if $(x_{s_n})_{n\in N}$ is a basis for $X$, where $N=\{n\in\N\mid s_n\in \theta\}$. We now show the following.

\begin{lemma}\label{lemageral}
Let $\theta\in\WF$, and let $X$ be a Banach space with an unconditional basis $(e_s)_{s\in\theta}$. Let $Y$ be an infinite dimensional subspace of $X$. Then $Y$ contains a basic sequence $(y_k)_{k\in\N}$ equivalent to a semi-normalized block sequence $(x_k)_{k\in\N}$ of $(e_s)_{s\in\theta}$ with completely incomparable supports.
\end{lemma}

Before we prove this lemma, let's show a simple lemma that will be important in our proof. We say that an operator $T:X\to Y$ is \emph{strictly singular} if for all infinite dimensional subspace $Z\subset X$, $T_{|Z}:Z\to Y$ is not an embedding.

\begin{lemma}\label{u8u}
Let $(X_1,\|\cdot\|_1),...,(X_n,\|\cdot\|_n)$ be Banach spaces, and let $Y\subset \oplus_{i=1}^n X_i$ be an infinite dimensional subspace. Consider the standard  projections $P_j:\oplus_{i=1}^nX_i\to X_j$, for all $j\in\{1,...,n\}$. Then, there exists $j\in\{1,...,n\}$ such that $P_j:Y\to X_j$ is not strictly singular.  
\end{lemma}

\begin{proof}
Let $X=\oplus_{i=1}^nX_i$.  As this is a finite sum, we can assume $X=(\oplus_{j=1}^nX_j)_{\ell_1}$, i.e., if $(x_1,...,x_n)\in X$, then $\|x\|_X=\sum_j\|x_j\|_{j}$.  Assume towards a contradiction that $P_j$ is strictly singular, for all $j\in\{1,...,n\}$. By a classic property of strictly singular operators (see \cite{D}, Proposition B.5), we know that for all $\eps>0$ there exists an infinite dimensional subspace $A\subset Y$ such that $\|P_{j|A}\|<\eps$, for all $j\in\{1,...,n\}$. Pick $x\in A$ of norm one. Then, as  $x=(P_1(x),...,P_n(x))$, we have $\|x\|_X\leqslant n\eps$. By choosing $\eps<1/n$ we get a contradiction. 
\end{proof}

\noindent\emph{Proof of Lemma \ref{lemageral}.}
For each $s\in\theta$, let $\Lambda_s=\{\lambda\in\N\mid s^\smallfrown(\lambda)\in \theta\}$, and enumerate each $\Lambda_s$, say $\Lambda_s=\{\lambda^i_s\mid i\in\N\}$. For each $s\in\theta$, let $\theta_{s}=\{\tau\in\theta\mid 
s\preceq\tau\}$. 

For each $s\in\theta$, let $P_s:X\to X$ be the projection of $X$ onto

$$\overline{\text{span}} \{e_\tau\mid s\preceq \tau\}.$$\\
\noindent  For each $s\in\theta$, and each $n\in\N$, consider the projections

\begin{align*}
Q_{s,n}:\ \ \ X\ \ &\to \oplus_{i=1}^nP_{s^\smallfrown{(\lambda^i_s)}}(X)\\
(a_\tau)_{\tau \in \theta}&\to (a_\tau)_{\tau\in\cup_{i=1}^n \theta_{s^\smallfrown{(\lambda^i_s)} } }.
\end{align*}\\
\indent\textbf{Claim:} There exists $s\in\theta$ such that  $P_s:Y\to X$ is not strictly singular, but $Q_{s,n}:Y\to \oplus_{i=1}^nP_{s^\smallfrown{(\lambda^i_s)}}(X)$ is strictly singular, for all $n\in\N$.\\

Let's assume the claim is true and finish the proof of the lemma. Indeed, if $P_s:Y\to X$ is not strictly singular, we can substitute $Y$ by an infinite dimensional $Z\subset Y$ such that $P_s:Z\to X$ is an isomorphism onto its image. Let $E=P_s(Z)$, and notice that 

$$E\subset \overline{\text{span}} \{e_\tau\mid s\preceq \tau\}.$$\\
\noindent By Lemma \ref{u8u}, we can actually assume that $E\subset \overline{\text{span}} \{e_\tau\mid s\prec \tau\}$. Hence, for all $x\in E$, we have that $x=\lim_n Q_{s,n}(x)$.\\

\textbf{Claim:} There exists a normalized sequence $(y_j)_{j\in\N}$ in $E$ such that $Q_{s,n}(y_j)\to 0$, as $j\to \infty$, $\forall n\in\N$.\\

Indeed, for all $n\in\N$, there exists a normalized sequence $(y^n_j)_{j\in\N}$ in $E$ such that  

$$\|Q_{s,n}(y^n_j)\|<1/j , \text{ \ for all\ } j\in\N.$$\\
\indent Let $(y_j)_{j\in\N}$ be the diagonal sequence of the sequences $(y^n_j)_{j\in\N}$, i.e., $y_j=y^j_j$, for all $j\in\N$. Say $M$ is the unconditional constant of $(e_s)_{s\in\theta}$. Then,  $m\leqslant n$ 
implies $\|Q_{s,m}(x)\|\leqslant M \|Q_{s,n}(x)\|$, for all $x\in E$. Hence, $(y_j)_{j\in\N}$ has the required property.

Say $(\eps_i)_{i\in\N}$ is a sequence of positive real numbers converging to zero. As $Q_{s,n}(x)\to x$, as $n\to\N$, for all 
$x\in E$, we can pick increasing sequences of natural numbers $(n_k)_{k\in\N}$,  and $(l_k)_{k\in\N}$ such that

\begin{enumerate}[(i)]
\item  $\|Q_{s,{l_k}}(y_{n_k})-y_{n_k}\|_\theta<\eps_k$, for all $k\in\N$, and
\item $\|Q_{s,{l_k}}(y_{n_{k+1}})\|_\theta<\eps_k$, for all $k\in\N$.
\end{enumerate}

For each $k\in \N$, let 

$$x_k=Q_{s,{l_k}}(y_{n_k})-Q_{s,{l_{k-1}}}(y_{n_k}).$$ \\
\noindent Choosing $(\eps_k)_{k\in\N}$ converging to zero fast enough, we have that $(x_k)_{k\in\N}$ is semi-normalized and, by the principle of small perturbations,  that $(y_{n_k})_{k\in\N}$ is equivalent to $(x_{k})_{k\in\N}$ (see \cite{AK}, \text{Theorem 1.3.9}). Clearly, $(x_k)_{k\in\N}$  has completely incomparable supports. As $Y$ contains a sequence equivalent to $(x_k)_{k\in\N}$, the proof is complete.

We now prove our first claim. Suppose the claim is false, i.e., suppose that for all $s\in\theta$ such that $P_s:Y\to X$ is not strictly singular, there exists $n\in\N$ such that $Q_{s,n}:Y\to \oplus_{i=1}^nP_{s^\smallfrown{(\lambda^i_s)}}(X)$  is not strictly singular. By Lemma \ref{u8u}, if $Q_{s,n}:Y\to \oplus_{i=1}^nP_{s^\smallfrown{(\lambda^i_s)}}(X)$  is not strictly singular, there exists $m\leqslant n$ such that 

$$P_{s^\smallfrown \lambda^m_s}=P_{s^\smallfrown \lambda^m_s}\circ Q_{s,n}:Y\to X$$\\
\noindent  is not strictly singular. Therefore, for all  $s\in \theta$  such that $P_s:Y\to X$ is not strictly singular, there exists $s'\succ s$ such that $P_{s'}:Y\to X$ is not strictly singular. 

Now notice that $P_\emptyset:Y\to X$ is not strictly singular, indeed, $P_\emptyset=Id$. Therefore, by applying the last paragraph $\omega$ times, we get a sequence $(s_n)_{n\in\N}$ such that $P_{s_n}:Y\to X$ is not strictly singular, and $s_n\prec s_{n+1}$, for all $n\in\N$. In particular, $s_n\in\theta$, for all $n$, absurd, as  $\theta$ is well-founded.\qed

\section{$\ell_p$-Baire sums.}\label{sectionlp}

We now deal with $\ell_p$-Baire sums of basic sequences, this tool will be crucial in many of our results in these notes. Fix a basic sequence $\cE=(e_n)_{n\in\N}$, and $p\in[1,\infty)$. Let us define a Borel function $\varphi:\text{Tr}\to\SB$ in the following manner. For each $\theta\in\text{Tr}$ and $x=(x(s))_{s\in\theta}\in c_{00}(\theta)$ we define

\begin{align*}
\left\|x\right\|_{\cE,p,\theta}=\sup\Big\{\Big(\sum_{i=1}^{n}\big\|\sum_{s\in I_i} x(s)e_{|s|}\big\|^p_{\cE}\Big)^{\frac{1}{p}}| \ n\in\N,\ I_1, ...,\ I_n &\text{ incomparable}\\ &\ \ \ \text{segments of 
}\theta\Big\},
\end{align*}\\
\noindent where $\|.\|_{\cE}$ is the norm of $\overline{\text{span}}\{\cE\}$. Define $\varphi_{\cE,p}(\theta)$ as the completion of $c_{00}(\theta)$ under the norm $\left\|.\right\|_{\cE,p,\theta}$.  The space $\varphi_{\cE,p}(\theta)$ is known as the \emph{$\ell_p$-Baire sum} of $\overline{\text{span}}\{\cE\}$ (indexed by $\theta$).

Similarly, we define $\|.\|_{\mathcal{E},0,\theta}$ as

$$\left\|x\right\|_{\mathcal{E},0,\theta}=\sup\Big\{\big\|\sum_{s\in I} x(s)e_{|s|}\big\|_{\cE}\mid  I \text{ segment of }\theta\Big\},$$\\
\noindent and let $\varphi_{\mathcal{E},0}(\theta)$  be the completion of $(c_{00}(\theta),\|.\|_{\mathcal{E},0,\theta})$. We denote by $(e_s)_{s\in\theta}$ the sequence in $c_{00}(\theta)$ such that, for each $\tau\in\theta$, the coordinate $e_s(\tau)$ equals $1$ if $s=\tau$ and zero otherwise. Considering a compatible enumeration of $\N^{<\N}$, as in Section \ref{sectionlema}, the sequence $(e_s)_{s\in\theta}$ is clearly a basis for $\varphi(\theta)$.

 Pick $Y\subset C(\Delta)$ such that $\varphi_{\mathcal{E},p}(\NN)$ is isometrically isomorphic to $Y$. If we consider the natural isometries of $\varphi_{\mathcal{E},p}(\theta)$ into $\varphi_{\cE,p}(\NN)$, we can see $\varphi_{\cE,p}$ as a Borel function from $\Tr$ into $\SB$. This gives us the following (see \cite{S}, Proposition 3.1, page 79).

\begin{prop}
Let $p\in[1,\infty)$ or $p=0$. Then, the  map $\varphi_{\mathcal{E},p}:\Tr\to\SB$ defined above is Borel. 
\end{prop}

The following lemma summarizes the main properties of the $\ell_p$-Baire sum that we will need later in these notes.

\begin{lemma}\label{arroto}
Let $\mathcal{E}$ be a basic sequence. The Borel function $\varphi_{\mathcal{E},p}:\Tr\to\SB$ defined above has the following properties:

\begin{enumerate}[(i)]
\item If $\theta\in\IF$, then $\varphi_{\mathcal{E},p}(\theta)$ contains $\overline{\text{span}}\{\mathcal{E}\}$.
\item If $\theta\in\WF$, then $\varphi_{\mathcal{E},p}(\theta)$ is $\ell_p$-saturated, i.e., every infinite dimensional subspace of $\varphi_{\mathcal{E},p}(\theta)$ contains an isomorphic copy of $\ell_p$.
\end{enumerate}

\noindent The analogous is true for $\varphi_{\mathcal{E},0}:\Tr\to\SB$, i.e., 

\begin{enumerate}[(i)]
\item If $\theta\in\IF$, then $\varphi_{\mathcal{E},0}(\theta)$ contains $\overline{\text{span}}\{\mathcal{E}\}$.
\item If $\theta\in\WF$, then $\varphi_{\mathcal{E},0}(\theta)$ is $c_0$-saturated, i.e., every infinite dimensional subspace of $\varphi_{\mathcal{E},0}(\theta)$ contains an isomorphic copy of $c_0$.
\end{enumerate}

\end{lemma}

\begin{proof} If $\theta\in\IF$, clearly $\varphi_{\mathcal{E},p}(\theta)$ contains $\overline{\text{span}}\{\mathcal{E}\}$.  Indeed, let $\beta$ be a branch of $\theta$, then $\overline{\text{span}}\{\mathcal{E}\}\cong\varphi_{\mathcal{E},p}(\beta)\emb\varphi_{\mathcal{E},p}(\theta)$, where by $\varphi_{\mathcal{E},p}(\beta)$ we mean $\varphi_{\mathcal{E},p}$ applied to the tree $\{s\in\NN\mid s<\beta\}$.

Say $\theta\in\WF$, and let $E$ be an infinite dimensional subspace of  $\varphi_{\mathcal{E},p}(\theta)$. By Lemma \ref{lemageral}, $E$ has a basic sequence equivalent  to a semi-normalized block sequence $(x_k)_{k\in\N}$ with completely incomparable supports. It is trivial to check that a semi-normalized block sequence with completely incomparable supports is equivalent to the $\ell_p$-basis (resp. $c_0$-basis). So we are done.
\end{proof}

Let $\mathcal{P}\subset \SB$. We say that $\mathcal{P}$ is a \emph{class} of Banach spaces if $\cP$ is closed under isomorphism, i.e.,  for all $X,Y\in\SB$, $X\in\mathcal{P}$ and $Y\cong X$ imply $Y\in \mathcal{P}$. We say that a class of Banach spaces $\mathcal{P}\subset \SB$ is \emph{pure} if $X\in\mathcal{P}$ implies $Y\in \mathcal{P}$, for all subspace $Y\subset X$. We say a class $\mathcal{P}\subset\SB$ is \emph{almost-pure} if for all $X\in \mathcal{P}$ and all infinite dimensional $Y\subset X$ there exists an infinite dimensional subspace $Z\subset Y$ such that $Z\in\mathcal{P}$. 

\begin{thm}\label{lema}
Say $\mathcal{P}\subset \SB$ is almost-pure and that $\ell_p$ (reps. $c_0$) does not embed in any $Y\in \mathcal{P}$, for some $p\in[1,\infty)$. Then $\text{C}_\mathcal{P}=\{Y\in\SB\mid \exists Z\in \cP,\ Z\emb Y\}$ is $\Sigma^1_1$-hard. In particular, the same is true if $\mathcal{P}$ is pure and does not contain $\ell_p$ (reps. $c_0$), for some $p\in[1,\infty)$. 
\end{thm}

\begin{proof}
This is a simple application of Lemma \ref{arroto}. Indeed, let $\cE$ be a basis for $C(\Delta)$, and consider the restriction of $\varphi_{\cE,p}$ to the set of infinite trees, say $\text{ITr}$. It is easy to see that $\text{ITr}$ is Borel (see \cite{S}, Proposition 1.6, page 72), so  ${\varphi_{\cE,p}}_{|\text{ITr}}$ is a Borel function. By Lemma \ref{arroto},  ${\varphi_{\cE,p}}_{|\text{ITr}}:\text{ITr}\to \SB$ is a Borel reduction from $\IF$ to  $\text{C}_\cP$. Therefore, as $\IF$ is $\Sigma^1_1$-hard, $\text{C}_\cP$ is $\Sigma^1_1$-hard.  
\end{proof}

\section{Descriptive complexity of the inevitable classes.}\label{newnew}

\subsection{Spaces containing a hereditarily indecomposable subspace.}\label{masha}

In 1991 W. T Gowers and B. Maurey independently solved the unconditional basic sequence problem, i.e., they constructed a Banach space with no unconditional basic sequence  (see \cite{GM}). It was noticed by W. B. Johnson that the space constructed by Gowers and B. Maurey not only had no unconditional basic sequence but was also hereditarily indecomposable. 

We say that an infinite dimensional Banach space $X$ is \emph{hereditatily indecomposable} if none of $X$ subspaces can be decomposed as a direct sum of two infinite dimensional subspaces. Clearly, the class $\text{HI}=\{X\in\SB\mid X\text{  is hereditarily} \allowbreak\text{indecomposable}\}$ is a pure class and it contains no $\ell_p$. Hence, Theorem \ref{lema} gives us the following.

\begin{cor}\label{hhii}
$\text{C}_{\text{HI}}$ is $\Sigma^1_1$-hard.
\end{cor}

We come back to the class of hereditarily indecomposable spaces in Subsection \ref{hiGM}, where we  show  that the set \text{HI} is complete coanalytic, and that $\text{C}_{\text{HI}}$ is at most $\Sigma^1_2$.

\subsection{Spaces containing a continuously tight subspace.}\label{sectiontight}

In \cite{FR}, Ferenczi and Rosendal defined a new class of Banach spaces, the class of continuously tight spaces, and proved many interesting properties about this class. In Theorem 3.13 of \cite{FR}, for example, Ferenczi and Rosendal have shown that every Banach space must contain either a minimal subspace or a continuously tight subspace, giving us another dichotomy for Banach spaces.

Denote by $[\N]$ the set of increasing sequences of natural numbers. We can see $[\N]$ as a Borel subset of $\N^\N$. A basic sequence $(e_n)_{n\in\N}$ is called \emph{continuously tight} if there exists a continuous function $f:bb(e_n)\to [\N]$ such that, for all block basis $\bar{y}=(y_n)_{n\in\N}\in bb(e_n)$, if we set $I_j=\{m\in\N\mid f(\bar{y})_{2j}\leq m\leq f(\bar{y})_{2j+1}\}$, then for all infinite set $A\subset \N$, 

$$\overline{\text{span}}\{\bar{y}\}\not\hookrightarrow\overline{\text{span}}\{e_n\mid n\not\in \cup_{j\in A} I_j\},$$\\
\noindent i.e., $\overline{\text{span}}\{\bar{y}\}$ does not embed into $\overline{\text{span}}\{e_n\}$ avoiding an infinite number of the intervals $I_j$. A space with a continuously tight basis is called \emph{continuously tight}. The Tsirelson space is an example of a continuously tight Banach space (see \cite{FR}, Corollary 4.3).

 For a detailed study of continuously tight spaces and other related properties (e.g., tight spaces, tight with constants, tight by range, etc) see \cite{FR}. Let $\text{CT}=\{X\in\SB\mid X\text{ is continuoulsy tight}\}$ be our coding for the class of continuously tight separable Banach spaces. 
 
\begin{cor}\label{cu}
$\text{C}_\text{CT}$ is $\Sigma^1_1$-hard.
\end{cor}

\begin{proof}
\text{Proposition 3.3} of \cite{FR} says that continuously tight spaces contain no minimal subspaces, and \text{Theorem 3.13} of \cite{FR} says that a space with no minimal subspaces contains a continuously tight subspace. Therefore,  $\text{CT}$ is an almost-pure class. Also, again by Proposition 3.3 of \cite{FR}, we have that no elements of $\text{CT}$ contain an isomorphic copy of $\ell_p$. Hence, Theorem \ref{lema} gives us that $\text{C}_\text{CT}$ is $\Sigma^1_1$-hard.
\end{proof}

We cannot obtain any lower bound for the complexity of the set of continuously tight Banach spaces with the method of $\ell_p$-Baire sums. We come back to the class of continuously tight spaces in Subsection \ref{sectiontightmesmo}, where we  show  that the set \text{CT} is $\Pi^1_1$-hard by using a different method.

\subsection{Spaces containing a minimal subspace.}\label{sectionminimal}

A Banach space $X$ is called \emph{minimal} if every infinite dimensional subspace of $X$ contains an isomorphic copy of $X$. We now turn our attention to the following: Although $\text{M}=\{X\in\SB\mid X\text{ is}\allowbreak \text{ minimal}\}$ is clearly a pure class, $\text{M}$ contains $c_0$ and $\ell_p$, for all $p\in[1,\infty)$. Therefore, \text{Theorem \ref{lema}} does not say anything about the  complexity of $\text{C}_\text{M}$. However, we can use the  construction of Tsirelson space by T. Figiel and W. B. Johnson (see \cite{FJ}, Section $2$) in order to construct a $\varphi:\Tr\to\SB$ that will solve our problem. 

Fix a compatible enumeration of $\NN$, i.e., $(s_n)_{n\in\N}$ such that $s_n\preceq s_m$ implies $n\leqslant m$ and for all $s\in\NN$ there exists $n\in\N$ such that $s_n=s$. This enumeration give us an order on $\N^{<\N}$. With this ordering in mind, we say that $E_0<E_1$ if $\max E_0<\min E_1$, for all  finite sets $E_0, E_1\subset \NN$. We write $k<E$ if $\{s_k\}<E$. 

Given  $\theta\in\Tr$, $E\subset\theta$, and $x=\sum_{n\in\N}a_{s_n}e_{s_n}\in c_{00}(\theta)$ (for some $(a_{s_n})_{n\in\N}\in\R^\N$), we let $Ex=\sum_{s_n\in E}a_{s_n}e_{s_n}$.

We define a Borel function $\varphi:\text{Tr}\to\SB$ as,  for each $\theta\in\text{Tr}$ and each $x=(x(s))_{s\in\theta}\in c_{00}(\theta)$, let $(\|.\|_{\theta,m})_{m\in\N}$ be inductively defined by

\begin{align*}
\|x\|_{\theta,0}&=\|x\|_0,\text{ and }\\
\|x\|_{\theta,m+1}&=\max\{\|x\|_0,\ds{\frac{1}{2}}\max\sum_{i=1}^k\|E_ix\|_{\theta,m}\},
\end{align*}

\noindent for all $m\in\N$, where the ``inner$"$ maximum above is taken over all $k\in\N$ and all completely incomparable finite sets $(E_i)_{i=1}^k$ ($E_i\subset \theta$, for all $i\in\{1,...,k\}$) such that $k\leqslant E_1<...<E_k$ (for the definition of completely incomparable sets of a tree, see Section \ref{definition}). Exactly as we have for the standard Tsirelson space, we can define a norm $\|.\|_\theta$ as 

$$\|x\|_\theta=\underset{m\to\infty}{\lim}\|x\|_{\theta,m},$$\\
\noindent for all $x\in c_{00}(\theta)$. We define $\varphi(\theta)$ to be the completion of $c_{00}(\theta)$ under this norm. This norm can be implicitly defined as

$$\left\|x\right\|_\theta=\max\{\|x\|_0,\ds{\frac{1}{2}}\max\sum_{i=1}^k\|E_ix\|_\theta\},$$\\
\noindent where the ``inner$"$ maximum above is taken over all $k\in\N$ and all completely incomparable finite sets $(E_i)_{i=1}^k$ ($E_i\subset \theta$, for all $i\in\{1,...,k\}$) such that $k\leqslant E_1<...<E_k$. 

By the universality of $C(\Delta)$ for separable Banach spaces, we can identify $\varphi(\NN)$ with an isometric copy inside of $C(\Delta)$. As we can identify each $\varphi(\theta)$ with a subspace of $\varphi(\NN)$  in a natural fashion, we can see $\varphi$ as a Borel function from $\Tr$ to $\SB$ (see \cite{S}, Proposition 3.1, page 79, for similar arguments).

\begin{thm}\label{lemaminimal}
Let $\varphi:\Tr\to\SB$ be defined as above. Then $\varphi$ is a Borel function with the following properties

\begin{enumerate}[(i)]
\item $c_0\emb\varphi(\theta)$, for all $\theta\in\IF$, and
\item $\varphi(\theta)$ has the following property if $\theta\in\WF$: for every infinite dimensional subspace $E\subset\varphi(\theta)$, there exists a further subspace $F\subset E$ isomorphic to an infinite dimensional subspace of Tsirelson space, i.e., $\varphi(\theta)$ is Tsirelson-saturated.
\end{enumerate}

\noindent In particular,  $\text{C}_\text{M}$ is $\Sigma^1_1$-hard, and  $\text{C}_{c_0}$ cannot be Borel separated from the set of Tsirelson-saturated Banach spaces.
\end{thm}

Before  proving this theorem, notice the following trivial consequence of Lemma \ref{u8u}.

\begin{lemma}\label{u8uu}
A finite sum of  spaces satisfying property $\text{(ii)}$ of Theorem \ref{lemaminimal} still has property $\text{(ii)}$.
\end{lemma}

\noindent\emph{Proof of Theorem \ref{lemaminimal}.}
If $\theta\in\IF$, it is clear that $c_0\emb\varphi(\theta)$. Say $\theta\in\WF$, let us show that every infinite dimensional subspace of $\varphi(\theta)$ contains a subspace isomorphic to an infinite dimensional subspace of Tsirelson's space. Say $E\subset \varphi(\theta)$ is an infinite dimensional subspace.

As $\theta\in\WF$, Lemma \ref{lemageral} gives us that $E$ contains a sequence equivalent to a  block sequence $(y_n)_{n\in\N}$ of $\varphi(\theta)$ with completely incomparable supports. We will be done once we prove the following claim.\\

\textbf{Claim:}  $(y_n)_{n\in\N}$ is equivalent to a subsequence of the standard basis of Tsirelson space.\\

Let $\|\cdot \|_{T,\theta}$ be the standard Tsirelson norm on $c_{00}(\theta)$, i.e.,    

$$\left\|x\right\|_{T,\theta}=\max\{\|x\|_0,\ds{\frac{1}{2}}\max\sum_{i=1}^k\|E_ix\|_{T,\theta}\},$$\\
\noindent where the ``inner$"$ maximum above is taken over all $k\in\N$ and all finite sets $(E_i)_{i=1}^k$ ($E_i\subset \theta$, for all $i\in\{1,...,k\}$) such that $k\leqslant E_1<...<E_k$. The only difference between $\|\cdot\|_\theta$ and $\|\cdot \|_{T,\theta}$ is that in $\|\cdot\|_{T,\theta}$ we do not have the restriction of $(E_i)_{i=1}^k$  being completely incomparable. 

It is clear that the the basis $(e_s)_{s\in\NN}$ of the completion of $(c_{00}(\theta),\|\cdot \|_{T,\theta})$ is equivalent to the standard basis of Tsirelson space. Moreover, the basis $(e_s)_{s\in\theta}$ of the completion of $(c_{00}(\theta),\|\cdot \|_{T,\theta})$ is equivalent to a subsequence of the standard  basis of Tsirelson space, if $\theta$ is infinite (this because this norm has no dependance on the structure of the tree $T$). Also, we clearly have

\begin{align}\label{uperb}
\|\sum_{i=1}^ka_iy_i\|_\theta\leqslant\|\sum_{i=1}^ka_iy_i\|_{T,\theta},
\end{align}\\
\noindent for all $a_1,...,a_k\in\R$. Mimicking the proof of \text{Lemma II.1} of \cite{CS} we have the following lemma.

\begin{lemma}\label{leamm}
Let $(p_n)_{n\in\N}$ be a increasing sequence of natural numbers. Let $(e_{s_n})_{n\in\N}$ be the standard basis of $\varphi(\theta)$. Let $y_n=\sum_{i=p_n+1}^{p_{n+1}}b_ie_{s_i}$ (for all $n\in\N$) be a normalized block sequence of $(e_{s_n})_{n\in\N}$ and assume $(y_n)_{n\in\N}$ has completely incomparable supports. Then

$$\|\sum_{n\in\N}a_ne_{s_{p_n+1}}\|_\theta\leqslant\|\sum_{n\in\N}a_ny_n\|_\theta,$$\\
\noindent for any sequence of scalars $(a_n)_{n\in\N}$.
\end{lemma}

\begin{proof}
It is enough to show that, for any sequence $(a_n)_{n\in\N}$, we have 

$$\|\sum_{n\in\N}a_ne_{s_{p_n+1}}\|_{\theta,m}\leqslant\|\sum_{n\in\N}a_ny_n\|_\theta,$$\\
\noindent for all $m\in\N$. Let us proceed by induction on $m\in\N$. For $m=0$ the result is clear. Assume the equation above holds for a fixed $m\in\N$. Let $x=\sum_{n\in\N}a_ne_{p_n+1}$ and $y=\sum_{n\in\N}a_ny_n$. Fix $k\in\N$, and completely incomparable finite sets $(E_n)_{n=1}^k$ such that $k\leqslant E_1<...<E_k$. Consider the sum

$$\ds{\frac{1}{2}}\sum_{i=1}^k\|E_ix\|_{\theta,m}.$$\\
\noindent Since the support of $x$ is contained in $\{p_n+1\mid n\in\N\}$, we may assume that 

$$E_i\subset \{p_n+1\mid n\in\N\},$$\\
\noindent for all $i\in\N$. Applying the inductive hypothesis, we have

$$\ds{\frac{1}{2}}\sum_{i=1}^k\|E_ix\|_{\theta,m}\leqslant \ds{\frac{1}{2}}\sum_{i=1}^k\|\sum_{\substack{n\in\N\\ p_n+1\in E_i}}a_ny_n\|_{\theta}$$\\
\noindent As $p_n+1\in E_i$ implies $k\leqslant p_n+1$, and as $(y_n)_{n\in\N}$ has completely incomparable supports, the sum on the right hand side of the equation above is allowed as an ``inner$"$ sum in the definition of the norm $\|.\|_{\theta}$. Therefore, we have 

$$\ds{\frac{1}{2}}\sum_{i=1}^k\|E_ix\|_{\theta,m}\leqslant\|y\|_{\theta},$$\\
\noindent for all $k\in\N$, and all completely incomparable finite sets $(E_n)_{n=1}^k$  such that $k\leqslant E_1<...<E_k$. Hence, for any sequence of scalars $(a_n)_{n\in\N}$, we have

$$\|\sum_{n\in\N}a_ne_{s_{p_n+1}}\|_{\theta,m+1}\leqslant\|\sum_{n\in\N}a_ny_n\|_\theta,$$\\
\noindent and we are done.
\end{proof}

Let $(b_n)_{n\in\N}$ be the sequence of scalars such that our block sequence $(y_n)_{n\in\N}$ can be written as  $y_n=\sum_{i=p_n+1}^{p_{n+1}}b_ie_{s_i}$. Then, \text{Lemma \ref{leamm}} gives us that

\begin{align}\label{eq2}
\|\sum_{n\in\N}a_ne_{s_{p_n+1}}\|_\theta\leqslant\|\sum_{n\in\N}a_ny_n\|_\theta,
\end{align}\\
\noindent for any sequence of scalars $(a_n)_{n\in\N}$. As the supports of $(y_n)_{n\in\N}$ are completely incomparable we have that 

\begin{align}\label{eq3}
\|\sum_{n\in\N}a_ne_{s_{p_n+1}}\|_\theta=\|\sum_{n\in\N}a_ne_{s_{p_n+1}}\|_{T,\theta},
\end{align}\\
\noindent for any sequence of scalars $(a_n)_{n\in\N}$. By \text{Proposition II.4} of \cite{CS}, we have 

\begin{align}\label{lowerb}
\ds{\frac{1}{18}}\|\sum_{n\in\N}a_ny_n\|_{T,\theta}\leqslant\|\sum_{n\in\N}a_ne_{s_{p_n+1}}\|_{T,\theta},
\end{align}\\ 
\noindent for all sequence of scalars $(a_n)_{n\in\N}$. Therefore, putting Equation \ref{uperb}, Equation \ref{eq2}, Equation \ref{eq3}, and Equation \ref{lowerb} together, we have that

$$\|\sum_{i=1}^ka_ne_{s_{p_n+1}}\|_{T,\theta}\leqslant\|\sum_{i=1}^ka_ny_n\|_\theta\leqslant18\|\sum_{i=1}^ka_ne_{s_{p_n+1}}\|_{T,\theta},$$\\
\noindent for all  $a_1,...,a_k\in\R$.  Hence, the sequence $(y_n)_{n\in\N}$ as a sequence in $\varphi(\theta)$ is equivalent to the sequence $(e_{s_{p_n+1}})_{n\in\N}$ as a sequence in the Tsirelson space (the completion of $(c_{00}(\theta), \|\cdot\|_{T,\theta})$).

To conclude that $\varphi(\theta)$ contains no minimal  subspaces if $\theta\in\WF$, recall that Tsirelson space contains no minimal subspaces (see \cite{CS}, Corollary VI.b.6), so we are done.\qed\\

It is easy to see, by simply counting quantifiers, that the set $\text{C}_\text{M}$ is at most $\Sigma^1_3$, i.e., the Borel image of a set that can be written as the complement of the Borel image of a coanalytic set. However, by using Gowers' theorem, Rosendal was able to find a better upper bound for $\text{C}_\text{M}$ (see \cite{R}, Appendix). 

\begin{thm}\textbf{(C. Rosendal)}
$\text{C}_\text{M}$ is $\Sigma^1_2$.
\end{thm}

\begin{problem}
What is the exact complexity of $\text{C}_\text{M}$? Is it $\Sigma^1_2$-complete? 
\end{problem}

In Subsection \ref{outrodia}, we talk a little bit about Rosendal's proof for $\text{C}_\text{M}$ being $ \Sigma^1_2$. We will notice that Rosendal's  proof also gives us that \text{M} is at most $\Delta^1_2$.

\subsection{Continuously tight Banach spaces.}\label{sectiontightmesmo}

We are now capable of giving a lower bound for the complexity of CT, the set of continuously tight spaces. 

Recall, a Banach space $X$ with a basis $(x_k)_{k\in\N}$ is said to be \emph{strongly asymptotic $\ell_p$} if there exists a function $f:\N\to\N$ and a constant $C>0$ such that any set of $m$ unit vectors in $\overline{\text{span}}\{x_k\mid k\geq f(m)\}$ with disjoint supports is $C$-equivalent to the basis of $\ell_p^m$. The Tsirelson space with its standard basis is an example of an strongly asymptotic $\ell_1$ space (see \cite{CS}, Chapter V). Also, a Banach space $X$ is said to be \emph{crudely finitely representable} in a Banach space $Y$ if there exists $M>0$ such that every finite subspace of $X$ $M$-embeds into $Y$.

The proof below is an adaptation of the proof of Proposition 4.2,  in \cite{FR}.

\begin{thm}\label{tighttighthu}
The set of continuously tight spaces \text{CT} is $\Pi^1_1$-hard. Moreover, the set of continuously tight basis, say $\mathcal{CT}$, is $\Pi^1_1$-hard. 
\end{thm}

\begin{proof}
Let $\varphi:\Tr\to \SB$ be the map in Lemma \ref{lemaminimal}. As $c_0$ embeds into $\varphi(\theta)$ if $\theta\in\IF$, we only need to show that $\varphi(\theta)$ is continuously tight if $\theta\in\WF$.  Indeed, as the map $\theta\in\Tr\mapsto (e_{s})_{s\in\theta}\in\cB$ is a Borel map, this is enough to prove both assertions of  the theorem.

 A Banach space $X$ with a basis is said to be \emph{tight with constants} if no Banach space embeds uniformly into its tail subspaces (see \cite{FR}, Proposition 4.1). Let's show that $\varphi(\theta)$ is tight with constants, for all $\theta\in\WF$. As spaces which are tight with constants are also continuously tight (see Proposition \ref{versa} below) we will be done.

Assume towards a contradiction that, for some $K>0$, there exists a space $Y$ which $K$-embeds into all tail subspaces of $\varphi(\theta)$. By Theorem \ref{lemaminimal}, we can assume, by taking a subspace, that $Y$ is generated by a sequence $(y_k)_{k\in\N}$ which is equivalent to a subsequence of the basis of Tsirelson space. Therefore, $(y_k)_{k\in\N}$ is unconditional and strongly asymptotic $\ell_1$. Let $C>0$ and $f:\N\to \N$ be as in the definition of strongly asymptotic $\ell_1$ spaces.

By Proposition 1 of \cite{Jo}, we have that for all $m\in\N$, there exists $N(m)\in\N$ such that $(y_1,...,y_m)$ is $2K$-equivalent to a sequence of vectors in the linear span of $N(m)$ disjointly supported unit vectors in any tail of $Y$. In particular, in the tail $\overline{\text{span}}\{y_k\mid k\geq f(N(m))\}$. Therefore, as $Y$ is strongly asymptotic $\ell_1$, we have that $(y_1,...,y_m)$ $2KC$-embeds into $\ell_1$, for all $m\in\N$. So $Y$ is crudely finitely representable in $\ell_1$, and therefore $Y$ embeds into $
L_1$ (see \cite{AK}, Theorem 11.1.8). Hence, as $(y_k)_{k\in\N}$ is unconditional asymptotic $\ell_1$, we have that $Y$ contains $\ell_1$ (see \cite{DFKO}, Proposition 5), absurd, because $Y$ is a subspace of Tsirelson space.
\end{proof}

\begin{prop}\label{versa}
Let $X$ be a Banach space with basis $(e_n)_{n\in\N}$. Say $(e_n)_{n\in\N}$ is tight with constants (see definition in the proof above). Then $(e_n)_{n\in\N}$ is continuously tight. 
\end{prop}

\begin{proof}
For this, we will use details of the proof of Proposition 4.1 of \cite{FR}. Precisely, let $X\in\SB$ be a Banach space with a basis $(e_n)_{n\in\N}$ which is tight with constants. For each $L\in\N$, let $c(L)>0$ be a constant such that if two block sequences of $(e_n)_{n\in\N}$ differ from at most $L$ terms, then they are $c(L)$-equivalent.  For each $(y_n)_{n\in\N}\in bb(e_n)$, let us define a sequence $(I_j)_{j\in\N}$ of finite intervals of natural numbers. 

By Proposition 4.1 of \cite{FR}, for each $K,m\in\N$, there exists an $l>m$ such that 

\begin{align}\label{bedbug}
\text{span}\{y_n\mid m\leq n\leq l\}\not\hookrightarrow_K\overline{\text{span}}\{e_n\mid n\geq l\}.
\end{align}\\
\noindent Let $l_1\in\N$ be the minimal $l\in\N$ as above, for $m=1$, and $K=c(1)$. Let $I_1=[1,l_1]$, where if $a\leq b\in\N$, $[a,b]=\{n\in\N\mid a\leq n\leq b\}$. Assume we had already defined finite intervals $I_1<...<I_{j-1}\subset \N$ and numbers $l_1<...<l_{j-1}\in\N$. Define $l_{j}$ as  the minimal $l\in\N$ as in (\ref{bedbug}) above, for $m=\max\{I_{j-1}\}+1$, and $K=j\cdot c(\max\{I_{j-1}\}+1)$. Let $I_j=[\max\{I_{j-1}\}+1,l_j]$. 

By the proof of Proposition 4.1 of \cite{FR}, we have that, for all $K\in\N$,

$$\text{span}\{y_n\mid n\in I_K\}\not\hookrightarrow_K \overline{\text{span}}\{e_n\mid n\not\in I_K\}.$$\\
\noindent In particular, for all infinite set $A\subset \N$, 

$$\overline{\text{span}}\{y_n\}\not\hookrightarrow\overline{\text{span}}\{e_n\mid n\not\in \cup_{j\in A}I_j\}.$$\\
\noindent Hence, we had defined a map $\bar{y}=(y_n)\mapsto (I_j)_{j\in\N}$, and  we will be done if this assignment is continuous, i.e., if there exists a continuous function $f:bb(e_n)\to [\N]$ such that $I_j=[f(\bar{y})_{2j},f(\bar{y})_{2j+1}]$. For this, we only need to notice that in order to obtain a finite chunk of the sequence of intervals  $(I_j)_{j\in\N}$, say $I_1,...,I_K$, we only need to know a finite chunk of $(y_n)_{n\in\N}$, precisely, $y_1,...,y_{\max\{I_K\}}$. So we are done.
\end{proof}

\begin{prop}
$\text{CT}$, $\text{C}_\text{CT}$ and $\mathcal{CT}$ are at most $\Sigma^1_2$. 
\end{prop}

\begin{proof}
This is a simple matter of counting quantifiers and the fact that we only quantify over standard Borel spaces in the definition of those three classes. Indeed, for $\mathcal{CT}$, for example, we have

\begin{align*}
(e_n)_{n\in\N}\in \mathcal{CT}\Leftrightarrow\ &\exists \text{ continuous }\ f:bb(e_n)\to[\N],\ \forall (y_n)_{n\in\N}\in bb(e_n),\\ &\forall\text{ infinite }\ A\subset \N,\ \forall (x_n)_{n\in\N}\in \text{span}\{e_n\mid n\not\in \cup_{j\in A} I_j\}^\N,\\ &\forall K\in\N,\ (x_n)_{n\in\N}\not\approx_K (y_n)_{n\in\N}.
\end{align*}\\
\indent The only quantifier that demands some explanation is $``\exists \text{ continuous }\ f:bb(e_n)\to[\N]"$. For this, let $\N^{[<N]}$ be the set of finite increasing sequence of natural numbers. Then it is easy to see that a continuous function $f:bb(e_n)\to[\N]$ gives us a function $g:fbb(e_n)\to\N^{[<N]}$ such that $g(\bar{y})\preceq g(\bar{x})$, if $\bar{y}\preceq\bar{x}$, and vice versa, where $fbb(e_n)$ is the set of finite block sequences of $(e_n)_{n\in\N}$ (see Section \ref{definition}). So, as $fbb(e_n)$ is countable, the space of functions $fbb(e_n)\to\N^{[<N]}$ is a standard Borel space, so we are done. The same arguments work for $\text{CT}$ and  $\text{C}_\text{CT}$.
\noindent 
\end{proof}

\subsection{Mininal spaces.}\label{outrodia}

It follows straight forward from the definition of minimal Banach spaces that $\text{M}=\{X\in\SB\mid X\text{ is minimal}\}$ is $\Pi^1_2$. In this subsection we show that $\text{M}$ is also $\Sigma^1_2$. Hence, $\text{M}$ is at most $\Delta^1_2$. 

Using Gowers' theorem (see \cite{Go2}), and a corollary of the solution of the distortio problem (see \cite{OS}, and \cite{OS2}), Rosendal had shown (see \cite{R}, Appendix) that if a Banach space $X$ not containing $c_0$ contains a minimal subspace then there exists a basic sequence $(e_n)_{n\in\N}$ in $X$, a block subsequence $(y_n)_{n\in\N}\in bb(e_n)$, and a continuous function $f: bb(y_n)\to \overline{\text{span}}\{e_n\}^\N$ such that, for all $\bar{z}=(z_n)_{n\in\N}\in bb(y_n)$, we have

$$\overline{\text{span}}\{f(\bar{z})\}\subset \overline{\text{span}}\{z_n\}\ \ \text{ and }\ \  (e_n)_{n\in\N}\approx f(\bar{z}).$$\\
\indent By counting quantifiers, the set of Banach spaces satisfying the property above is at most $\Sigma^1_2$. Clearly, if a Banach space satisfies the property above, then it contains a minimal subspace, indeed, $\overline{\text{span}}\{y_n\}$ is minimal. As the set of Banach spaces containing $c_0$ is $\Sigma^1_1$ ($\Sigma^1_1$-complete actually), this gives us that $\text{C}_\text{M}$ is at most $\Sigma^1_2$.  

We now notice that this also gives us an equivalent characterization of minimality. Indeed,  Rosendal's result clearly implies that if   $X$ is a minimal Banach space not containing $c_0$ then there  exists a basic sequence $(e_n)_{n\in\N}$ in $X$, a block subsequence $(y_n)_{n\in\N}\in bb(e_n)$, such that

$$X\hookrightarrow\overline{\text{span}}\{y_n\},$$\\
\noindent  and there exists a continuous function $f: bb(y_n)\to \overline{\text{span}}\{e_n\}^\N$ such that, for all $\bar{z}=(z_n)_{n\in\N}\in bb(y_n)$, we have

$$\overline{\text{span}}\{f(\bar{z})\}\subset \overline{\text{span}}\{z_n\}\ \ \text{ and }\ \  (e_n)_{n\in\N}\approx f(\bar{z}).$$\\
\indent By counting quantifiers, the set of Banach spaces satisfying the property above is at most $\Sigma^1_2$.  Notice that if a Banach space satisfies the property above, then it is minimal. Therefore, as the set of minimal Banach spaces containing $c_0$ is the set of spaces that embed into $c_0$, and as this set is easily seen to be analytic, we have that $\text{M}$ is at most $\Sigma ^1_2$. As M is also $\Pi^1_2$, we have the following.

\begin{prop}
The class of minimal Banach spaces M is at most $\Delta^1_2.$
\end{prop}

\subsection{Hereditarily indecomposable spaces.}\label{hiGM}

Let $\text{HI}=\{X\in\SB\mid X\text{ is hereditarily }\allowbreak\text{indecomposable}\}$. In Subsection \ref{masha}, we proved, using the method of $\ell_p$-Baire sums (Corollary \ref{hhii}), that $\text{C}_\text{HI}$ is $\Sigma^1_1$-hard. However, this method does not allow us to obtain any information about the complexity of $\text{HI}$. In this subsection, we will use a more complex construction in order to compute the complexity of \text{HI}.

Precisely, we follow Argyros presentation (see \cite{AT}, Chapters II and III) of how to construct HI extensions of ground norms in order to define a Borel function $\varphi:\Tr\to\SB$ such that $\varphi^{-1}(\text{HI})=\WF_\infty$, where $\WF_\infty$ denotes the subset of infinite well-founded trees. This will show that $\text{HI}$ is $\Pi^1_1$-hard. As a Banach space $X$ is hereditarily indecomposable if, and only if, for all subspaces $Z,W\subset X$ and all $\eps>0$, there exist $z\in Z$ and $w\in W$ such that $\|z-w\|<\eps\|z+w\|$, we can easily show that $\text{HI}$ is coanalytic.  Therefore, we will show that $\text{HI}$ is complete coanalytic.

\begin{thm}\label{hhhiii}
\text{HI} is coanalytic.
\end{thm}

\begin{proof}
This is a simple consequence of the fact that $X\in\text{HI}$ if, and only if,  for all subspaces $Z,W\subset X$ and all $\eps>0$, there exists $z\in Z$ and $w\in W$ such that $\|z-w\|<\eps\|z+w\|$. Indeed, notice that

\begin{align*}
X\in\text{HI}\Leftrightarrow\ &\forall Z,W \subset X,\ \forall \eps\in\Q_+,\ \exists n,m\in\N \\
&\|S_{n}(Z)-S_{m}(W)\|<\eps\|S_{n}(Z)+S_{m}(W)\|,
\end{align*}\hfill

\noindent where ``$\forall Z,W \subset X"$ means ``for all subspaces $Z,W\subset X"$. As $\{(Z,W,X)\in\SB^3|Z,W\subset X\}$ is well known to be Borel (\cite{S}, see Lemma 1.9, page 73), we are done.
\end{proof}

This trivially gives us the following upper bound for the complexity of $\text{C}_\text{HI}$.

\begin{cor}
$\text{C}_\text{HI}$ is $\Sigma^1_2$.
\end{cor}

The only thing left to show is that $\text{HI}$ is $\Pi^1_1$-hard. For this, let us define a special Borel map $\varphi:\text{Tr}\to \SB$ such that $\varphi^{-1}(\text{HI})=\WF_\infty$. As the construction of such $\varphi$ will heavily rely on the construction of HI extensions of a ground norm, we tried to be consistent with Argyros notation (\cite{AT}, Chapters II and III). We believe this will make the presentation more clear for the reader which is familiar with HI spaces. Therefore, the notation used to denote some spaces in this subsection will be slightly different from the notation chosen in the rest of these notes. We will make sure to point out the differences as they appear though.  To start with, we will denote by $\mathfrak{X}(D_{G}{(\theta)})$ the resulting space $\varphi(\theta)$.

First,  we fix a compatible enumeration for $\NN$, say $(s_i)_{i\in\N}$. Let $(e_{s_i})_{i\in\N}$ be the standard unit basis of $c_{00}(\N^{<\N})$. 

Let  $x=(x(s))_{s\in\N^{<\N}}\in c_{00}(\N^{<\N})$, and $x^*=(x^*(s))_{s\in\N^{<\N}}\in c_{00}(\N^{<\N})$, we define

$$x^*(x)=\sum_{s\in\N^{<\N}}x(s)x^*(s),$$\\
\noindent i.e., the notation $ ``^*"$ means that we will consider $x^*$ as a functional on $c_{00}(\N^{<\N})$. For  $x^*=(x^*(s))_{s\in\N^{<\N}}\in c_{00}(\N^{<\N})$, we let $\text{supp}(x^*)=\{s\in\N^{<\N}\mid x^*(s)\neq 0\}$.

Let

$$G=\big\{\sum_{i=1}^n a_ie_{s_{j_i}}\mid n\in\N, s_{j_i}\in\N^{<\N},s_{j_1}\prec...\prec s_{j_n},|a_i|= 1\big\}.$$\\
\indent The set $G$ is called a \emph{ground set} (for a definition of ground sets see \cite{AT}, Definition II.1, page 21). We define $Y_{G}$ as the completion of $c_{00}(\N^{<\N})$ under the norm

$$\|x\|_{G}	=\sup\{g(x)\mid g\in G\},$$\\
\noindent for all $x\in c_{00}(\N^{<\N})$. For each $\theta\in \Tr$, we let $Y_{G}(\theta)$ denote the subspace of $Y_G$ generated by $\{e_s\mid s\in\theta\}$. So $Y_G(\N^{<\N})=Y_G$. We will define $\mathfrak{X}(D_{G}{(\theta)})$ as a ground norm extension of $Y_{G}{(\theta)}$.

\textbf{Remark:} Notice that, according to the notation of \text{Lemma} \ref{arroto}, it is clear that $Y_{G}{(\theta)}\cong \varphi_{\cE,0}(\theta)$, if $\cE$ is the standard $\ell_1$-basis. 

For each $j\in \N$, let $\cA_j=\{F\subset \N\mid |F|\leqslant j\}$, where $|F|$ is the cardinality of $F\subset \N$. As in Subsection \ref{sectionminimal}, for finite sets $E_1,E_2\subset \N$, we write $E_1<E_2$ if $\max E_1<\min E_2$. Let $g_1,...,g_n\in c_{00}(\N^{<\N})$. We write $g_1<...<g_n$ if $\text{supp} (g_1)<...<\text{supp} (g_n)$. Let $(g_l)_{l=1}^{n}$ be a finite sequence in $c_{00}(\N^{<\N})$ such that $g_1<...<g_n$, and $m\in\N$, we define the \emph{$(\cA_{n}, \frac{1}{m})$-operation on $(g_l)_{l=1}^{n}$}  as the functional $g=\frac{1}{m}(g_1+...+g_{n})$.

Fix two sequences of natural numbers $(m_j)_{j\in\N}$ and $(n_j)_{j\in\N}$ such that $m_1=2$, $m_{j+1}=m_j^5$, $n_1=4$, and $n_{j+1}=(5n_j)^s$, where $s_j=\log_2 m_{j+1}^3$.

We now define a norming set $D_{G}\subset c_{00}(\N^{<\N})$ that will give us the norm we will use to define $\mathfrak{X}(D_{G}{(\theta)})$, i.e., $\varphi(\theta)$.  In the definition below we use the term \emph{``$n_{2j-1}$-special sequence$"$}. As this definition will play no role in our proof of the theorem and as it is a really technical definition, we chose to omit it here. The interested reader can find the precise definition of an $n_{2j-1}$-special sequence in \cite{AT}, Chapter III, page $40$. 

Let $E\subset \N^{<\N}$, and $x^*=(x^*(s))_{s\in \N^{<\N}}\in c_{00}(\N^{<\N})$. We define $Ex^*=(x^*(s))_{s\in E}$.

\begin{defn}
We define $D_{G}$ as the minimal subset of $c_{00}(\N^{<\N})$ satisfying

\begin{enumerate}[(i)]
\item $G\subset D_{G}$.
\item $D_{G}$ is symmetric, i.e., $g\in D_{G}$ implies $-g\in D_{G}$.
\item $D_{G}$ is closed under the restriction of its elements to intervals of $\N^{<\N}$, i.e., if $E\subset\N^{<\N}$ is an interval and $g\in D_{G}$, then $Eg\in D_{G}$.
\item $D_{G}$ is closed under the $(\cA_{n_{2j}}, \frac{1}{m_{2j}})$-operations, i.e., if $(g_l)_{l=1}^{n_{2j}}$ is a sequence in $D_{G}$ such that $g_1<...<g_{n_{2j}}$, then $g=\frac{1}{m_{2j}}(g_1+...+g_{n_{2j}})$  belongs to $D_{G}$.
\item $D_{G}$ is closed under $(\cA_{n_{2j-1}}, \frac{1}{m_{2j-1}})$-operations on special sequences, i.e., for every $n_{2j-1}$-special sequence $(g_1,...,g_{n_{2j-1}})$ in $D_{G}$, the functional $g=\frac{1}{m_{2j-1}}(g_1+...+g_{n_{2j-1}})$  belongs to $D_{G}$.
\item $D_{G}$ is rationally convex.
\end{enumerate}
\end{defn}

We define a norm on  $c_{00}(\N^{<\N})$, as 

$$\|x\|_{D_G}=\sup\{g(x)\mid g\in D_{G}\},$$\\
\noindent for all $x\in c_{00}(\N^{<\N})$. Let $\mathfrak{X}(D_{G})$ be the completion of $c_{00}(\N^{<\N})$ under this norm. For each $\theta\in\Tr$, let $\mathfrak{X}(D_{G}(\theta))$ be the subspace of $\mathfrak{X}(D_{G})$ generated by $\{e_s\mid s\in\theta\}$. Therefore, for each $\theta\in\Tr$, we assign a space $\varphi(\theta)=\mathfrak{X}(D_{G}{(\theta)})$. 

Identify $\mathfrak{X}(D_{G}(\N^{<\N}))=\mathfrak{X}(D_{G})$ with one of its isometric copies in $\SB$. It is clear that the map $\varphi:\Tr\to \SB$ such that $\varphi(\theta)=\mathfrak{X}(D_{G}(\theta))$, is Borel (see \cite{S}, Proposition 3.1, page 79, for similar arguments).

\begin{thm}\label{ultimoo}
Let $\varphi:\Tr\to\SB$ be the function defined above. Then

\begin{enumerate}[(i)]
\item $\varphi(\theta)=\mathfrak{X}(D_{G}{(\theta)})$ contains $\ell_1$, for all $\theta\in\IF$.
\item $\varphi(\theta)=\mathfrak{X}(D_{G}{(\theta)})$ is hereditarily indecomposable, for all $\theta\in\WF_\infty$.
\end{enumerate}

\noindent In particular, $\text{HI}$ is $\Pi^1_1$-hard, and $\text{C}_{\ell_1}$ cannot be Borel separated from $\text{HI}$.
\end{thm}

\begin{proof} First, notice that if $\theta\in\IF$ then $\ell_1\emb\varphi(\theta)$. Indeed,  on segments $I \subset\theta$, the $\ell_1$-norm given by the ground set $G$ is greater than the norm given by $D_{G}$. So, if $\theta$ has a branch, say $\beta$, we have $\mathfrak{X}(D_{G}{(\beta)})\cong\ell_1$.

In order to show the second part of the theorem consider the ``identity$"$ map $Id:\mathfrak{X}(D_{G}{(\theta)})\to Y_{G}{(\theta)}$. We will show that $Id:\mathfrak{X}(D_{G}{(\theta)})\to Y_{G}{(\theta)}$ is strictly singular, for all $\theta\in\WF_\infty$. Once we do that, we will be done by \text{Theorem III.7} of \cite{AT} (page $42$). Indeed, it is clear from the proof of \text{Theorem III.7} of \cite{AT}, that we have the following.

\begin{thm}
Let $G$ be a ground set in $c_{00}(\N^{<\N})$, and let $Y_G$, and $\mathfrak{X}(D_{G})$ be the spaces obtained as above. If $Z\subset \mathfrak{X}(D_{G})$ is an infinite dimensional subspace, and the restriction $Id_{|Z}:Z\to Y_G$ is strictly singular, then $Z$ is hereditarily indecomposable.
\end{thm}

\begin{prop}\label{ss}
Let $\theta\in\WF_\infty$. Then $Id:\mathfrak{X}(D_{G}{(\theta)})\to Y_{G}{(\theta)}$ is strictly singular.
\end{prop}

\begin{proof}

Suppose not. Then there exists an infinite dimensional subspace $Y\subset \mathfrak{X}(D_{G}{(\theta)})$ such that $Id_{|Y}$ is an isomorphism with its image. We now look at $Y=Id(Y)\subset Y_{G}{(\theta)}$. By \text{Lemma} \ref{arroto} and the remark following the definition of $Y_{G}{(\theta)}$, there is a normalized sequence $(y_i)_{i\in\N}$ in $Y\subset Y_{D_{G}{(\theta)}}$ which is equivalent to the standard basis of $c_0$. In particular, there exists $C>0$ such that

$$\big\|\sum_{i=1}^ny_i\big\|_{G}<C,\ \text{ for all }n\in \N.$$\\
\indent Moreover, we can assume, by Lemma \ref{porque} below, that there exists a sequence $(g_i)_{i\in\N}$ in $G$ such that $g_i(y_i)>\frac{1}{2}$ and $g_i<g_{i+1}$, for all $i\in\N$, and $|g_i(y_k)|<2^{-(k+2)}$, for all $i\neq k$. Therefore, by the definition of the norm of $\mathfrak{X}(D_{G}{(\theta)})$, we have that 

\begin{align*}\big\|\sum_{i=1}^{n_{2j}}y_i\big\|_{D_G}
&\geq\frac{1}{m_{2j}}\sum_{i=1}^{n_{2j}}g_i(y_i)+\frac{1}{m_{2j}}\sum_{\substack{i,k=1\\ i\neq k}}^{n_{2j}}g_i(y_k)\\
&\geq\frac{n_{2j}}{2m_{2j}}-\frac{n_{2j}}{m_{2j}}\sum_{k=1}^\infty 2^{-(k+2)}\\
&=\frac{n_{2j}}{4m_{2j}}
\end{align*}\\
As $\frac{n_{2j}}{m_{2j}}\to\infty$, as $j\to\infty$, and as $Id_{|Y}$ is an isomorphism, we get a contradiction.
\end{proof}

The proof of Theorem \ref{ultimoo} is now done.
\end{proof}

\begin{thm}\label{hhhhiiii}
\text{HI} is complete coanalytic.
\end{thm}

In order to prove the result above we made use of the following lemma.

\begin{lemma}\label{porque}
Let $\varphi_{\mathcal{E},0}:\Tr\to\SB$ be as in Lemma \ref{arroto}. If $\theta\in\WF$, and  $Y\subset \varphi_{\mathcal{E},0}(\theta)$ is infinite dimensional, then there exists a normalized sequence $(y_i)_{i\in\N}$ in $Y$ equivalent to the $c_0$-basis. Moreover, there exists a sequence $(g_i)_{i\in\N}$ in $G(\theta)$ such that $g_i<g_{i+1}$, for all $i\in\N$, $g_i(y_i)>\frac{1}{2}$, for all $i\in\N$, and $|g_i(y_k)|<2^{-(k+2)}$, for all $i\neq k$.
\end{lemma}	
 

\begin{proof}
This can be obtained by a simple modification in the proof of Lemma \ref{lemageral}. For completeness, we write the modifications here. Assume all the notation in the proof of Lemma \ref{lemageral}. So we have $X=\varphi_{\eps,0}(\theta)$, and $s\in\theta$ is such that $P_s:Y\to X$ is not strictly singular, but $Q_{s,n}:Y\to X$ is strictly singular, for all $n\in\N$. Let $E=P_s(Z)$, where $Z\subset Y$ is a subspace such that $P_s:Z\to X$ is an isomorphism with its image. As in Lemma \ref{lemageral}, we can assume that

$$E\subset \overline{\text{span}} \{e_\tau \mid s\prec\tau\}.$$\\
\indent For each $m\in\N$, let $I_m:X\to X$ be the standard projection over the first $m$ coordinates of the basis $(e_s)_{s\in\theta}$, i.e., $I_m(\sum_{s\in\theta}a_se_s)=\sum_{i=1}^ma_{s_i}e_{s_i}$.  Let $(y_k)_{k\in\N}$ be a normalized sequence in $E$ such that $Q_{s,n}(y_k)\to 0$, as $k\to\infty$, for all $n\in\N$ (we showed such sequence exists in the proof of Lemma \ref{lemageral}). As $Q_{s,n}(x)\to x$, as $n\to\N$, for all 
$x\in E$, we have that given a sequence $(\eps_k)_{k\in\N}$ of positive real numbers, we can pick increasing sequences of natural numbers $(n_k)_{k\in\N}$, $(m_k)_{k\in\N}$, and $(l_k)_{k\in\N}$ such that

\begin{enumerate}[(i)]
\item  $\|Q_{s,{l_k}}(y_{n_k})-y_{n_k}\|_\theta<\eps_k$, for all $k\in\N$, and
\item $\|Q_{s,{l_k}}(y_{n_{k+1}})\|_\theta<\eps_k$, for all $k\in\N$.
\item$\|I_{m_k}\big(Q_{s,{l_k}}(y_{n_k})-Q_{s,{l_{k-1}}}(y_{n_k})\big)-\big(Q_{s,{l_k}}(y_{n_k})-Q_{s,{l_{k-1}}}(y_{n_k})\big)\|_\theta<\eps_k$, for all $k\in\N$.
\end{enumerate}

For each $k\in \N$, let 

$$x_k=I_{m_k}\big(Q_{s,{l_k}}(y_{n_k})-Q_{s,{l_{k-1}}}(y_{n_k})\big).$$ \\
\noindent Choosing $(\eps_k)_{k\in\N}$ converging to $0$ sufficiently fast, we have that $(x_k)_{k\in\N}$ is equivalent to $(y_{n_k})_{k\in\N}$, and, as $(x_k)_{k\in\N}$ has completely incomparable supports, it is easy to see that $(x_k)_{k\in\N}$ is also equivalent to the $c_0$-basis (as in Lemma \ref{arroto}). Also, by taking a subsequence if necessary, we can assume that $(x_k)_{k\in\N}$ is a block sequence. Hence, as we can assume $\|x_i\|_\theta>\frac{1}{2}$, for all $i\in\N$, there exists a sequence $(g_i)_{i\in\N}$ in $G(\theta)$ such that $g_i<g_{i+1}$, for all $i\in\N$, and $g_i(x_i)>\frac{1}{2}$, for all $i\in\N$, and $g_i(x_k)=0$, for all $i\neq k$. 

Clearly, we can assume $\text{supp}(g_i)\subset \text{supp}(x_i)$, for all $i\in\N$. Hence, if $(y'_i)_{i\in\N}$ is a sequence in $Z$ such that $P_s(y'_i)=y_{n_i}$, for all $i\in\N$, we have that  $g_i(y'_i)=g_i(y_{n_i})=g_i(x_i)$, for all $i\in\N$. Therefore,   $g_i(y'_i)>\frac{1}{2}$, for all $i\in\N$. 

On the other hand,  it is easy to see that $g_i(y'_k)=g_i(y_{n_k}-x_k)$, for all $i\neq k$. Hence, as $\|y_{n_k}-x_k\|_\theta<2\eps_k+\eps_{k-1}$, we can assume that $|g_i(y'_k)|<2^{-(k+2)}$, for all $i\neq k$. Although $(y'_k)_{k\in\N}$ is only semi-normalized, it is clear from the proof, that we can assume $(y'_k)_{k\in\N}$ is normalized, so we are done.
\end{proof}

\subsection{Unconditional basis.}\label{ubub} A basic sequence $(x_j)_{j\in\N}$ in a Banach space $X$ is unconditional if, and only if, there exists $M>0$ such that, for all $n\in\N$, for all $a_1,...,a_n\in\Q$, and all $b_1,...,b_n\in\Q$ such that $|a_j|\leqslant |b_j|$, for all $j\in\{1,...,n\}$, we have

$$\|\sum_{j=1}^ka_jx_j\|\leqslant \|\sum_{j=1}^kb_jx_j\|.$$\\
\noindent Therefore, it is clear that the set $\mathcal{UB}\subset \cB$ of unconditional basis is Borel, where $\cB$ is our coding for the basic sequences (see Section \ref{definition}).

We now consider instead of the set of unconditional basis, the set of Banach spaces with an unconditional basis, say $\text{UB}$, and the set of Banach spaces containing an unconditional basis, say $\text{C}_\text{UB}$. As 

\begin{align*}
X\in\text{UB} \Leftrightarrow\ &\exists (x_j)_{j\in\N}\in C(\Delta)^\N \text{ such that }\\
&(x_j)_{j\in\N}\text{ is an unconditional basis for }X,
\end{align*}\\
\noindent and the condition $``(x_j)_{j\in\N}\text{ is a basis for X}"$ is easily seen to be Borel, we have that $\text{UB}$ is analytic. Analogously, $\text{C}_\text{UB}$ is also analytic. 

Let us now give a lower bound for $\text{C}_\text{UB}$. As hereditarily indecomposable spaces cannot contain an unconditional basis, Theorem \ref{ultimoo} gives us the following.

\begin{thm}\label{ubca}
The set of separable Banach spaces containing an unconditional basis $\text{C}_{\text{UB}}$ is $\Sigma^1_1$-hard. Moreover,  $\text{C}_{\text{UB}}$ is $\Sigma^1_1$-complete.
\end{thm}

\begin{problem}
Is \text{UB} Borel? Is \text{UB} complete analytic? What about $S=\{X\in\SB\mid X\text{ has a Schauder basis}\}$? Similarly, as we have for UB, we can easily see that S is analytic. Is S Borel? Is S complete analytic? 
\end{problem}

\textbf{Acknowledgements:} The author would like to thank his adviser C. Rosendal
for all the help and attention he gave to this paper. Also, the author would like to thank Joe Diestel for reading this manuscript, and Spyros Argyros for his help in the HI part of these notes.

\end{document}